\documentclass[11pt, twoside]{amsart}
\usepackage[cp1251]{inputenc}
\usepackage{color}
\usepackage{amsthm,amssymb,latexsym}
\usepackage[centertags]{amsmath}
\usepackage[T2A]{fontenc}
\usepackage{graphs} 
\usepackage{MnSymbol}
\usepackage{bbm}
\usepackage{centernot}
\usepackage{mathrsfs}
\usepackage{amsthm}
\usepackage{amsfonts}
\usepackage{textcomp}
\usepackage{newlfont}

\newif\ifPDF
\ifx\pdfoutput\undefined\PDFfalse
\else \ifnum \pdfoutput > 0 \PDFtrue
        \else \PDFfalse
        \fi
\fi

\ifPDF
  \usepackage[pdftex]{xcolor, graphicx}
  \usepackage[colorlinks=true,linkcolor=blue, citecolor=blue]{hyperref}%

\else
  \usepackage{color}
  \usepackage[dvips]{graphicx}
  \usepackage[dvips]{hyperref}
\fi

\usepackage{tkz-graph}
\tikzset{EdgeStyle/.style = {->}}
\tikzset{LabelStyle/.style= {fill=yellow}}
\usetikzlibrary{shapes,snakes,calendar,matrix,backgrounds,folding}

\usepackage{bbm}

\usepackage[top=1in, bottom=1.25in, left=1.25in, right=1.25in]
{geometry}

\theoremstyle{definition}
\newtheorem{definition}{Definition}[section]
\theoremstyle{remark}
\newtheorem{example}[definition]{Example}

\newtheorem{remark}[definition]{Remark}

\theoremstyle{plain}
\newtheorem{thm}[definition]{Theorem}
\newtheorem{theorem}[definition]{Theorem}
\newtheorem{prop}[definition]{Proposition}
\newtheorem{lemma}[definition]{Lemma}

\newtheorem{claim}[definition]{Claim}

\def\ol{\overline}
\def\ov{\overline}

\def\wh{\widehat}

\newcommand{\N}{\mathbb N}
\newcommand{\Z}{\mathbb Z}

\newcommand{\be}{\begin{equation}}
\newcommand{\ee}{\end{equation}}
\newcommand{\ba}{\begin{aligned}}
\newcommand{\ea}{\end{aligned}}
\newcommand{\mc}{\mathcal}

\newcommand{\R}{{\mathbb R}}

\numberwithin{equation}{section}
\newcommand{\ignore}[1]{}

\begin{document}

\title{Invariant measures for reducible generalized Bratteli diagrams}

\author{Sergey Bezuglyi}
\address{Department of Mathematics,
University of Iowa, Iowa City, IA 52242-1419
USA, 
\newline
ORCID: 0000-0002-3212-1750}
\email{sergii-bezuglyi@uiowa.edu}

\author{Olena Karpel}
\address{AGH University of Krakow, Faculty of Applied Mathematics, al. Adama Mickiewicza~30, 30-059 Krak\'ow, Poland \&
B. Verkin Institute for Low Temperature Physics and Engineering,
47~Nauky Ave., Kharkiv, 61103, Ukraine
\newline
ORCID: 0000-0001-5988-3774}

\email{okarpel@agh.edu.pl}

\author{Jan Kwiatkowski}

\address{Faculty of Mathematics and Computer Science, Nicolaus Copernicus University, ul. Chopina 12/18, 87-100 Toru\'n, Poland}

\email{jkwiat@mat.umk.pl}



\date{}

\begin{abstract}
   
In 2010, Bezuglyi, Kwiatkowski, Medynets and Solomyak [Ergodic Theory
Dynam. Systems 30 (2010), no.4, 973-1007] found a complete description of the set of probability ergodic tail invariant measures
on the path space of a standard (classical) stationary reducible Bratteli diagram. It was shown that every
distinguished eigenvalue for the incidence matrix
determines a probability ergodic invariant measure.  
In this paper, we show that this result does not hold for stationary reducible generalized Bratteli diagrams. We consider classes of stationary and non-stationary reducible generalized Bratteli diagrams with infinitely many simple standard subdiagrams, in particular, with infinitely many odometers as subdiagrams. We characterize the sets of all probability ergodic invariant measures for such diagrams and study partial orders under which the diagrams can support a Vershik homeomorphism.  
\end{abstract}

\maketitle
\tableofcontents

\section{Introduction}

Bratteli diagrams provide an important tool for the study of dynamical systems in measurable, Cantor, and Borel dynamics. A generalized Bratteli diagram is a natural extension of the notion of a classical (standard) Bratteli diagram, where each level has a countably infinite set of vertices. While standard Bratteli diagrams are particularly useful in Cantor dynamics to describe the simplex of 
probability tail invariant measures and study other properties of dynamical systems (see e.g. surveys \cite{Durand2010, BezuglyiKarpel2016, BezuglyiKarpel_2020, Putnam2018}), generalized Bratteli diagrams are used to model non-compact Borel dynamical systems \cite{BezuglyiDooleyKwiatkowski_2006}. Recent papers developed the study of dynamical systems on generalized Bratteli diagrams \cite{BezuglyiJorgensen2022, Bezuglyi_Jorgensen_Sanadhya_2022, Bezuglyi_Jorgensen_Karpel_Sanadhya_2023}. It was shown that
generalized Bratteli diagrams have many interesting 
phenomena in comparison with the standard case. For example,
one can find stationary generalized Bratteli diagrams with either a unique probability invariant measure, or uncountable many probability invariant measures, or without such measures at all. 

This paper continues the study of generalized
Bratteli diagrams and focuses mostly on a class of \textit{stationary reducible generalized Bratteli 
diagrams}. We
consider tail invariant measures and the existence of a Vershik map on such diagrams. The case of irreducible generalized Bratteli diagrams was considered in \cite{BezuglyiJorgensen2022,  Bezuglyi_Jorgensen_Karpel_Sanadhya_2023}.

The importance of the case of stationary Bratteli diagrams
is based on the following facts. Stationary standard 
Bratteli diagrams provide models of substitution dynamical systems (minimal and non-minimal ones), see  
\cite{Durand_Host_skau_1999}, \cite{Bezuglyi_Kwiatkowski_Medynets_2009}. The case of 
stationary generalized Bratteli diagrams is more 
challenging to study. Right now, it is known that 
a class of substitution dynamical systems on an infinite
alphabet (considered first by \cite{Ferenczi_2006}) can be realized as a generalized Bratteli diagram 
\cite{Bezuglyi_Jorgensen_Sanadhya_2022}. 

Having such a duality between the diagrams and substitution
dynamical systems, one can answer the principal question 
about an explicit description of the set of ergodic
probability measures. 
In \cite{BezuglyiKwiatkowskiMedynetsSolomyak2010}, the 
authors found a transparent algorithm for the construction 
of ergodic probability measures. For this, one needs 
to find all distinguished Perron eigenvalues of the 
incidence matrix. Then using the corresponding eigenvector,
the values of the invariant measure on the cylinder sets are determined by a simple formula (see 
Theorem \ref{thm-BKMS2010} for details). 

It seems rather surprising, but the mentioned theorem 
does not hold for stationary generalized Bratteli diagrams even in the case when all simple subdiagrams are standard odometers. Instead, we state a new result that provides 
a required characterization of probability tail invariant 
measures. For this, 
we use the method developed in \cite{BezuglyiKarpelKwiatkowski2015,  AdamskaBezuglyiKarpelKwiatkowski2017}, where ergodic invariant measures on standard Bratteli diagrams were obtained by using a procedure of an extension from a subdiagram. In this paper, we show that the same method can be used for generalized Bratteli diagrams. Moreover, for the class of the so-called reducible Bratteli diagrams with infinitely many odometers, this method gives all probability ergodic invariant measures on the diagram.

The outline of the paper is as follows. Section~\ref{Sect:prelim} provides main facts concerning standard and generalized Bratteli diagrams, the procedure of measure extension, and the description of all probability ergodic invariant measures for reducible stationary standard Bratteli diagrams. In Section~\ref{Sect:meas_ext_general}, we briefly discuss 
the procedure of measure extension from a vertex subdiagram for generalized Bratteli diagrams. 
In Section~\ref{Sect:DIO_meas_general}, we introduce 
the notion of reducible generalized Bratteli diagrams with infinitely many odometers and focus on the study of tail invariant measures and 
their extensions. We give a complete classification of probability ergodic invariant measures on such diagrams. 
In Section~\ref{Sect:Examples_DIO}, we present various classes of stationary and non-stationary reducible generalized Bratteli diagrams with infinitely many odometers and characterize their sets of probability ergodic invariant measures. Section~\ref{Sect:Vmap} is devoted to different orders on reducible generalized Bratteli diagrams with infinitely many odometers. For different classes of orders on such diagrams, we study whether the corresponding Vershik map can be extended to a homeomorphism. Our main results are presented in Theorems~\ref{Thm:Meas-ext-general}, \ref{thm:All-erg-inv-meas-DIO}, \ref{Thm:DIO_a_i_all_inv_m}, \ref{Thm:non-stat-DIO-ex-erg-mu}, \ref{Thm:orders-DIO}.

\section{Preliminaries}\label{Sect:prelim}

In this section, we recall the main definitions and facts 
concerning invariant measures on standard and generalized stationary Bratteli diagrams. For more details 
see \cite{BezuglyiKwiatkowskiMedynetsSolomyak2010},  
\cite{BezuglyiKarpel2016}, 
\cite{BezuglyiJorgensen2022}, 
\cite{Bezuglyi_Jorgensen_Karpel_Sanadhya_2023}. 

Throughout the paper, we will use the standard notation $\mathbb N, \Z, \R,$  
$\mathbb{N}_0 = \N \cup \{0\}$ for the sets of numbers, 
and $|\cdot |$ for the cardinality of a set.

\subsection{Basic definitions on Bratteli diagrams}

\begin{definition}
A {\it (standard) Bratteli diagram} is an infinite graph $B=(V,E)$ such that the vertex
set $V =\bigsqcup_{i\geq 0}V_i$ and the edge set $E=\bigsqcup_{i\geq 0}E_i$
are partitioned into disjoint subsets $V_i$ and $E_i$, where

(i) $V_0=\{v_0\}$ is a single point;

(ii) $V_i$ and $E_i$ are finite sets for all $i$;

(iii) there exists a range map $r \colon E \rightarrow V$ and a source map $s \colon E \rightarrow V$ such that $r(E_i)= V_{i+1}$ and $s(E_i)= V_{i}$ for all $i \geq 1$.
\end{definition}

\begin{definition}\label{Def:generalized_BD} A 
\textit{generalized Bratteli diagram} is a graded graph 
$B = (V, E)$ such that the 
vertex set $V$ and the edge set $E$ can be partitioned $V = \bigsqcup_{i=0}^\infty  V_i$ and $E = 
\bigsqcup_{i=0}^\infty  E_i$ so that the following 
properties hold: 
\vspace{2mm}

\noindent $(i)$ For every $i \in \mathbb{N}_0$, the number of vertices at each level 
$V_i$ is countably infinite, and the set $E_i$
of all edges between $V_i$ and $V_{i+1}$ is countable.
\vspace{2mm}

\noindent $(ii)$ For every edge $e\in E$, we define the 
range and 
source maps $r$ and $s$ such that $r(E_i) = V_{i+1}$ and 
$s(E_i) = V_{i}$ for $i \in \N_0$.  

\vspace{2mm}

\noindent $(iii)$ For every vertex $v \in V \setminus V_0$, we 
have $|r^{-1}(v)| < \infty$.  
\end{definition} 

Let $B = (V, E)$ be a standard or generalized Bratteli diagram. We will call the set $V_i$ the $i$th level of the diagram $B$. For generalized Bratteli diagrams, we will identify each $V_i$ with $\N$. Consider a 
finite or infinite sequence of edges $(e_i: e_i\in E_i)$ such 
that $s(e_i)=r(e_{i-1})$ which is called a finite or infinite 
path, respectively. We denote the set of infinite
paths starting at $V_0$ by $X_B$ and call it the \textit{path 
space}. For a finite path $\ol e = (e_0, ... , e_n)$, we denote 
$s(\ol e) = s(e_0), r(\ol e) = r(e_n)$. The set 
$$
    [\ol e] := \{x = (x_i) \in X_B : x_0 = e_0, ..., x_n = e_n\}, 
$$ 
is called the \textit{cylinder set} associated with $\ol e$. The {topology} on the path space $X_B$ is generated by cylinder sets. The path space $X_B$ is a zero-dimensional Polish space.

For vertices $v \in V_m$ and $w \in V_{n}$, we will 
denote 
by $E(v, w)$ the set of all finite paths between $v$ and $w$. Set $f^{(i)}_{v,w} = |E(v, w)|$ for every $w \in V_i$ and $v \in 
V_{i+1}$. In such a way,  we associate with the Bratteli diagram $B = (V,E)$ 
a sequence of non-negative matrices $(F_i)$, 
$i \in \N_0$ 
(called the \textit{incidence  matrices}) given by
\begin{equation}\label{Notation:f^i}
    F_i = (f^{(i)}_{v,w} : v \in V_{i+1}, w\in V_i),\ \   
    f^{(i)}_{v,w}  \in \N_0.
\end{equation} 
In the case of a generalized Bratteli diagram, incidence matrices are countably infinite.

For $w \in V_n$ and $n \in \N_0$, denote 
$$
X_w^{(n)} = \{x = (x_i)\in X_B : s(x_{n}) = w\}.
$$
The collection $(X_w^{(n)})_w$ of all such sets forms a partition 
$\zeta_n$ of $X_B$ into  
\textit{Kakutani-Rokhlin towers}
corresponding to the vertices from $V_{n}$. For $w \in V_n$ and $v_0 \in
V_0$, we set 
$$
h^{(n)}_{v_0, w} = |E(v_0, w)| 
$$ and define 
$$
H^{(n)}_w = \sum_{v_0 \in V_0} h^{(n)}_{v_0, w}, \ \ n \in \N.
$$ 
Set $H^{(0)}_w = 1$ for all $w\in V_0$.
This gives us the vector $H^{(n)} = (H^{(n)}_{w} : 
w \in V_n)$ associated with every level $n\in \N_0$. Since
$H^{(n)}_w = |E(V_0, w)|$, we call
$H^{(n)}_w$ the \textit{height of the tower} $X_w^{(n)}$ 
corresponding to the vertex $w\in V_n$.

\begin{definition}\label{Def:telescoping} Given a standard or generalized 
Bratteli diagram $B = (V,E)$ and a  monotone increasing sequence 
$(n_k : k \in \N_0)$  with $n_0 = 0$, we define a new Bratteli 
diagram $B' = (V', E')$ as follows: the vertex sets are 
determined by $V'_k = V_{n_k} $, and the edge sets  $E'_k = 
E_{n_{k}} \circ ...\circ E_{n_{k+1}-1}$ are formed by finite paths 
between the levels $V'_k$ and 
$V'_{k+1}$. The diagram $B' = (V', E')$  is called a 
\textit{telescoping} of the original diagram $B = (V,E)$. 
\end{definition}

\begin{remark}
    Notice that each telescoping of a generalized Bratteli diagram is again a generalized Bratteli diagram.
\end{remark}

\begin{definition}\label{Def:irreducible_GBD} (1) 
Let $B = B(F_n)$, where $(F_n)_n$ are the incidence matrices of $B$, be
a standard or generalized Bratteli diagram. If  $F_n = F$ for every $n \in 
\N$ (for standard diagrams) or every $n \in \mathbb{N}_0$ (for generalized diagrams), then the diagram $B$ is called \textit{stationary.} We 
will write $B = B(F)$ in this case.
 Unless stated otherwise, we will assume that every standard Bratteli diagram has a ``simple hat'', which means that there is a single edge from the vertex $v_0$ to every vertex $v \in V_1$.

(2) A standard Bratteli diagram is called \textit{simple} if there exists a telescoping
$B'$ of $B$ such that all entries of the incidence matrices of $B'$ are positive.  Since every vertex of a generalized Bratteli diagram can have only finitely many incoming edges, the notion of a simple Bratteli diagram cannot be applied to generalized diagrams.

(3) A generalized Bratteli diagram $B =(V,E)$ is called  \textit{irreducible} if 
for any vertices $i, j \in V_0$ and any level $V_n$ there exist 
$m 
> n$ and a finite path connecting $i \in V_n$ and $j \in V_m$. In 
other words, the $(j, i)$-entry of the matrix $F_{m-1} \cdots 
F_n$ 
is non-zero. Otherwise, the diagram is called \textit{reducible}.

\end{definition}

We will consider \textit{tail-invariant measures}
on the path space $X_B$ of a standard or generalized Bratteli diagram
$B$. By term \textit{measure} we always mean a non-atomic positive Borel measure.

\begin{definition}\label{Def:Tail_equiv_relation}
Let $B$ be a standard or generalized Bratteli diagram.
Two paths $x= (x_i)$ and $y=(y_i)$ in $X_B$ are called 
\textit{tail equivalent} if there exists  $n \in \mathbb{N}_0$ 
such that $x_i = y_i$ for all $i \geq n$. This notion defines a 
countable Borel equivalence relation $\mathcal R$ on the
path space $X_B$ (every equivalence class is countable and $\mathcal{R}$ is a Borel subset of $X_B \times X_B$) which is called the \textit{tail equivalence 
relation}.
A measure $\mu$ on 
$X_B$ is called \textit{tail-invariant} if, for any cylinder sets
$[\ol e]$ and $[\ol e']$ such that $r(\ol e) = r(\ol e')$, we have
$\mu([\ol e]) = \mu([\ol e'])$.
\end{definition}

\begin{thm}\label{BKMS_measures=invlimits}
 Let $B = (V,E)$ be a Bratteli diagram (generalized or standard) 
 with the sequence of incidence matrices $(F_n)$. Then:
\begin{enumerate}

\item Let  $\mu$ be a tail-invariant measure on $B$ which takes finite values on cylinder sets. For every 
$n\in \N_0$, define the vector $\ol p^{(n)} = \langle p^{(n)}_w : w \in V_n \rangle$, where 
\begin{equation}\label{eq:def_p_n}
    p^{(n)}_w= \mu([\ov e(w)]), \ \ w\in V_n. 
\end{equation} 
where $[\ov e(w)]$ is a cylinder set which ends in the vertex $w$.
Then the vectors   $\ol p^{(n)}$ satisfy the relation 
\begin{equation}\label{eq:formula_p_n}
F_n^{T} \ol p^{(n+1)} =\ol p^{(n)}, \quad n\geq 0.
\end{equation}

\item Conversely, suppose that $\{\ol p^{(n)}= (p_w^{(n)}) 
\}_{n \in \N_0}$ is a sequence of non-negative vectors such 
that $F_n^{T}\ol p^{(n+1)} =\ol p^{(n)}$ for all $n \in 
N_0$. Then there exists a uniquely determined tail 
invariant measure $\mu$ such that $\mu([\ov e(w)])= 
p_w^{(n)}$ for 
$w\in V_n$ and $n \in \mathbb N_0$.
\end{enumerate}

\end{thm} 

The {proof} of Theorem \ref{BKMS_measures=invlimits} 
is 
straightforward and can be found in 
\cite{BezuglyiKwiatkowskiMedynetsSolomyak2010} (for 
classical Bratteli diagrams) and 
\cite{BezuglyiJorgensen2022} (for generalized Bratteli 
diagrams).

To define a dynamical system on the path space of a 
generalized 
Bratteli diagram, we need to take a linear order $>$ on each
(finite) set $r^{-1}(v),\ v
\in V\setminus V_0$. This order defines a partial order $>$ on 
the sets of edges $E_i,\ i=0,1,...$, where edges
$e,e'$ are comparable if and only if $r(e)=r(e')$. 

\begin{definition}\label{Def:Ordered_GBD}
A generalized Bratteli diagram $B=(V,E)$
together with a partial order $>$ on $E$ is called 
\textit{an ordered generalized Bratteli diagram} $B=(V,E,>)$. 
\end{definition}

We call a (finite or infinite) path $e= (e_i)$ 
\textit{maximal (respectively minimal)} if every
$e_i$ is maximal (respectively minimal) among all 
elements from $r^{-1}(r(e_i))$. Denote by $X_{max}$ ($X_{min}$) the sets of all infinite maximal (minimal) paths in $X_B$.

\begin{definition}\label{Def:VershikMap}  For an ordered 
generalized Bratteli diagram $B=(V,E,>)$, we define a Borel 
transformation 
$\varphi_B : X_B \setminus X_{max} \rightarrow X_B \setminus X_{min}$
as follows.  Given $x = (x_0, x_1,...)\in X_B\setminus X_{max}$, 
let $m$ be the smallest number such that $x_m$ is not maximal. Let 
$g_m$ be the successor of $x_m$ in the finite set $r^{-1}(r(x_m))$.
Then we set $\varphi_B(x)= (g_0, g_1,...,g_{m-1},g_m,x_{m+1},...)$
where $(g_0, g_1,..., g_{m-1})$ is the minimal path in $E(V_0, 
s(g_{m}))$. The map $\varphi_B$ is
a Borel bijection. Moreover, $\varphi_B$ is a homeomorphism from 
$X_B\setminus X_{max}$ onto $X_B\setminus X_{min}$. If  
$\varphi_B$ admits a bijective Borel extension 
to the 
entire path space $X_B$, then we call the Borel transformation 
$\varphi_B : X_B  \rightarrow X_B$ a \textit{Vershik map}, and 
the 
Borel dynamical system $(X_B,\varphi_B)$ is called a generalized 
\textit{Bratteli-Vershik} system.

\end{definition}

\begin{remark}
    If the cardinalities of $X_{max}$ and $X_{min}$ are the same then there always exists a Borel extension of $\varphi_B$ to the whole path space $X_B$. 
\end{remark}

 In general, every measure $\mu$ that is invariant with respect to a Vershik map is also tail-invariant. If the sets of $X_{min}$ and $X_{max}$ have zero
measure, then we can identify tail-invariant measures with 
measures invariant with respect to a Vershik map.


\subsection{Subdiagrams and measure extension}
\label{subs subdiagrams} 
In this subsection, we give the basic definitions and results on subdiagrams of standard and generalized Bratteli 
diagrams. We also describe the notion of measure extension. We use 
the approach developed first in 
\cite{BezuglyiKarpelKwiatkowski2015} and
\cite{AdamskaBezuglyiKarpelKwiatkowski2017}.

Let $B = (V, E)$ be a standard or generalized Bratteli diagram. Consider nonempty subsets  $\ol V\subset V$ and $\ol E \subset E$ that can be written as $\ol V = \bigcup_n \ol V_n$ and $\ol E = \bigcup_n \ol E_n$, 
where $\ol V_n \subset V_n$ and $\ol E_n \subset E_n$.
Then we say that the pair $\ol B = (\ol V,\ol E)$ defines 
a {\em subdiagram} of $B$ if $\ol V = s(\ol E)$ and $s(\ol E)  = r(\ol E) \cup \ol V_0$. 

To define a {\em vertex} subdiagram of $B$, we begin with a sequence $\overline W = \{W_n\}_{n>0}$ of proper nonempty subsets $W_n$ of $V_n$, and set $W'_n = V_n \setminus W_n \neq \emptyset$ for all $n$. The vertex subdiagram $\overline B = (\ol W, \ol E)$ is formed 
by the vertices from $W_n$ and by the set of edges $\ol E_n$ 
whose source and range are in $W_{n}$ and $W_{n+1}$, respectively. Thus, the incidence  matrix $\ol F_n$ of $\ol B$ 
has the size $|W_{n+1}| \times |W_n|$, and it is represented 
by a block of $F_n$ corresponding to the vertices from 
$W_{n}$ and $W_{n+1}$. 

The path space $X_{\ol B}$ of a vertex subdiagram $\ol B$ is 
a closed subset of $X_B$. 
On the other hand, there are closed subsets of $X_B$ which are not obtained as the path space of a Bratteli subdiagram. It was proved in \cite{GiordanoPutnamSkau2004} 
 that a closed subset $Z \subset X_B$ is the path space of a subdiagram if and only if $\mathcal R|_{Z \times Z}$ is an etal\'{e} equivalence relation (see \cite{GiordanoPutnamSkau2004} for details).

Below we explain the procedure of measure extension from a subdiagram. This procedure was considered earlier in \cite{BezuglyiKarpelKwiatkowski2015} and
\cite{AdamskaBezuglyiKarpelKwiatkowski2017} for standard Bratteli diagrams, but it works also for generalized diagrams. The diagram $B$ below can be either a standard or generalized Bratteli diagram.
Let $\widehat X_{\ol B}:= \mathcal R(X_{\ol B})$ be the subset of all paths in $X_B$ that are tail equivalent to paths from $X_{\ol B}$. In other words, $\widehat X_{\ol B}$ is the
smallest $\mathcal R$-invariant subset of $X_B$ containing 
$X_{\ol B}$.
Let $\ov \mu$ be an ergodic tail invariant probability measure on $X_{\ol B}$. Then $\ol \mu$ can be canonically extended to the ergodic measure $\widehat{\ov \mu}$ on the space $\widehat X_{\ol B}$ by tail invariance, see \cite{BezuglyiKwiatkowskiMedynetsSolomyak2013}, 
\cite{AdamskaBezuglyiKarpelKwiatkowski2017}. More specifically, let the measure 
$\ol \mu$ be defined by a sequence of positive vectors 
$\{\ol p^{(n)} : n \in \N_0\}$ satisfying Theorem \ref{BKMS_measures=invlimits}, that is 
$\ol F_n(\ol p^{(n+1)}) = 
\ol p^{(n)}, n \in \N_0$, where $\ol F_n$ is the incidence
matrix for the subdiagram $\ol B$. 
Then, for every cylinder 
set $[\ol e] \subset X_B$ with $r(\ol e) = v \in \ol V_n$, we
set $\wh {\ol \mu}([\ol e]) = \ol p^{(n)}_v$. Then 
$\wh {\ol \mu}$ is defined on all clopen sets and it
can be finally extended to a Borel measure on of $X_B$.

Let $\ol B$ be a vertex subdiagram of a generalized Bratteli
diagram $B$ defined by a sequence of subsets $(W_i)$. 
Denote by $\wh X_{\ol B}^{(n)}$ the set of all paths $x = (x_i)_{i = 1}^{\infty}$ from $X_B$ such that the finite path  $(x_1, \ldots, x_n)$ ends at a vertex $v$ of $\ov B$, and the tail  $(x_{n+1},x_{n+2},\ldots)$ belongs to $\overline{B}$,
i.e., 
\begin{equation}\label{n-th level}
\wh X_{\ol B}^{(n)} = \{x = (x_i)\in \wh X_{\ol B} : r(x_i) \in W_i, \ \forall i \geq n\}.
\end{equation}
It is obvious that $\wh X_{\ol B}^{(n)} \subset \wh X_{\ol B}^{(n+1)}$, $\wh X_{\ol B} = \bigcup_n \wh X_{\ol B}^{(n)}$,
and 
\begin{equation}\label{extension_method}
\widehat{\ov \mu}(\wh X_{\ol B}) = \lim_{n\to\infty} \widehat{\ov \mu}(\wh X_{\ol B}^{(n)}) = 
\lim_{n\to\infty}\sum_{w\in W_n}  H^{(n)}_w \ov p^{(n)}_w.
\end{equation}
This limit can be finite or infinite. If it is finite, then we 
say that $\ol\mu$ admits a finite measure extension 
$\widehat{\ov \mu}(\wh X_{\ol B}) < \infty$.
To obtain an ergodic invariant measure on the whole path space $X_B$, we set $\widehat{\ov \mu}(X_B \setminus \wh X_{\ol B}) = 0$.


\subsection{Invariant measures on stationary standard Bratteli diagrams}

In this subsection, we recall the explicit description of all probability ergodic invariant measures on stationary standard Bratteli diagrams, the exposition is based on \cite{BezuglyiKwiatkowskiMedynetsSolomyak2010}.

Let $B$ be a stationary standard Bratteli diagram with $N$ vertices on each level $n \geq 1$. We identify every set of vertices $V_n$ for $n \geq 1$ with the set $\{1, \ldots, N\}$. Denote by $F$ the corresponding incidence matrix and let $A = F^T$. 

One can associate to $A$ a directed graph $G(A)$ with the vertices $\{1, \ldots, N\}$ such that there is an arrow from $i$ to $j$ if and only if $a_{ij} > 0$. 
We will say that vertices $i$ and $j$ are equivalent if either $i = j$ or there are paths in $G(A)$ from $i$ to $j$ and from $j$ to $i$. Denote by $\mathcal{E}_i$, $i = 1, \ldots, m$ the corresponding equivalence classes. Then each class $\mathcal{E}_i$ defines a submatrix $A_i$ of $A$ obtained by restricting $A$ to the set of vertices from $\mathcal{E}_i$.
Identify the family of sets $\{\mathcal{E}_i\}_{i = 1}^m$ with the set $\{1, \ldots, m\}$, and define the partial order on $\{1, \ldots, m\}$ as follows: for $\alpha, \beta \in \{1, \ldots, m\}$, we have $\beta \succeq \alpha$ if either $\alpha = \beta$ or there is a path in $G(A)$ from a vertex in $\mathcal{E}_{\beta}$ to a vertex in $\mathcal{E}_{\alpha}$. In this case, we say that class $\beta$ has access to class $\alpha$. If $\beta \succeq \alpha$ and $\beta \neq \alpha$, we write $\beta \succ \alpha$. The partial order $\succ$ defines the reduced directed graph $R(A)$ of $G(A)$ as follows: the set of vertices of $R(A)$ is $\{1,\ldots, m\}$, there is an edge from a vertex $\beta$ to a vertex $\alpha$ if and only if $\beta \succ \alpha$.
One can enumerate the vertices of $B$ and classes in $\{1,\ldots, m\}$
in such a way that $A$ assumes a block triangular form and $\beta \succ \alpha$ if $\alpha > \beta$ (with the usual ordering on integers):
$$
A = \begin{pmatrix}
    A_1 & 0 & \ldots & 0 & Y_{1, s+1} & \ldots & Y_{1,m}\\
    0 & A_2 & \ldots & 0 & Y_{2, s+1} & \ldots & Y_{2, m}\\
    \vdots & \vdots & \ddots & \vdots & \vdots & \ddots & \vdots\\
    0 & 0 & \ldots & A_s & Y_{s, s+1} & \ldots & Y_{s, m}\\
    0 & 0 & \ldots & 0 & A_{s+1} & \ldots & Y_{s+1,m}\\
    \vdots & \vdots & \ddots & \vdots & \vdots & \ddots & \vdots\\
     0 & 0 & \ldots & 0 & 0 & \ldots & A_m\\
\end{pmatrix},
$$
where $s \geq 1$, the square matrices $\{A_{i}\}_{i = 1}^s$ on the main diagonal are non-zero irreducible matrices, each square matrix $\{A_{i}\}_{i = s+1}^m$ is either irreducible or a $1 \times 1$ zero matrix. For any $j = s+1, \ldots, m$, at least one of the matrices $\{Y_{k,j}\}_{k = 1}^{j-1}$ is non-zero. Moreover, the matrix $Y_{k,j}$ is non-zero if and only if there is an edge in $R(A)$ from $k$ to $j$. 
Further, we telescope $B$ so that each non-zero matrix $A_i$ on the main diagonal is strictly positive. For more details see for instance \S 4.4 in \cite{LindMarkus1995}, where $R(F)$ is called the graph of communicating classes.

Let $\rho_{\alpha}$ be a spectral radius (a Perron eigenvalue) of $A_{\alpha}$. A vertex (class) $\alpha \in \{1, \ldots, m\}$ is called \textit{distinguished} if $\rho_{\alpha} > \rho_{\beta}$ whenever $\beta \succ \alpha$. In particular, vertices $\{1, \ldots, s\}$ are distinguished vertices in $R(A)$ and $A_{\alpha} \neq 0$ for a distinguished class $\alpha$. Let $\ov B_{\alpha}$ be a simple stationary subdiagram of $B$ generated by vertices that belong to $\mathcal{E}_{\alpha}$.
We call a real number $\lambda$  a \textit{distinguished eigenvalue} for $A$ if there exists a non-negative eigenvector $\textbf{x}$ with $A\textbf{x} = \lambda \textbf{x}$. A real number $\lambda$ is a distinguished eigenvalue if and only if there is a distinguished class $\alpha$ in $R(A)$ such that $\rho_{\alpha} = \lambda$. The corresponding non-negative eigenvector $(\xi_{\alpha}(1), \ldots, \xi_{\alpha}(N))^T$ is unique (up to scaling) and $\xi_{\alpha}(i) > 0$ if and only if $i$ has access to $\alpha$ (see \cite{Victory1985, Schneider1986, TamSchneider2001}). The vector $(\xi_{\alpha}(1), \ldots, \xi_{\alpha}(N))^T$ it is called a \textit{distinguished eigenvector} corresponding to a distinguished eigenvalue $\lambda_{\alpha}$.

We sum up the results from \cite{BezuglyiKwiatkowskiMedynetsSolomyak2010} in the following theorem:

\begin{theorem} \label{thm-BKMS2010}
\cite{BezuglyiKwiatkowskiMedynetsSolomyak2010}
Let $B$ be a stationary standard Bratteli diagram and let $A$ be an $N \times N$ matrix which is the transpose of the incidence matrix.
Then every probability ergodic invariant measure on $X_B$ corresponds to a distinguished class of vertices in $R(A)$, a distinguished right eigenvector for $A$, and a corresponding distinguished eigenvalue for $A$. Conversely, every distinguished class of vertices in $R(A)$,  distinguished right eigenvector and a corresponding distinguished eigenvalue for $A$ generate a probability ergodic invariant measure on $B$ in the following way:
if $\mu_{\alpha}$ is a probability ergodic invariant measure corresponding to a distinguished class of vertices $\alpha$ in $R(A)$ then $\mu_{\alpha}$ is (up to constant multiple) the extension of a unique invariant measure $\ov \mu_{\alpha}$ from $\ov B_{\alpha}$. If 
$(\xi_{\alpha}(1), \ldots, \xi_{\alpha}(N))^T$ and $\lambda_{\alpha}$ are the corresponding distinguished eigenvector and eigenvalue then 
$$
\mu_{\alpha}([\ov e(v_0,w)]) = \frac{\xi_{\alpha}(w)}{\lambda_{\alpha}^{n-1}}, 
$$
where $[\ov e(v_0,w)]$ is a cylinder set corresponding to a finite path $\ov e(v_0,w)$ which
ends in a vertex $w$ on level $n \geq 1$. 
\end{theorem}

\subsection{Invariant measures on stationary irreducible generalized Bratteli diagrams}

In \cite{BezuglyiJorgensen2022, Bezuglyi_Jorgensen_Karpel_Sanadhya_2023}, the authors used the Perron-Frobenius theory for infinite matrices to describe tail 
invariant measures on the path space of a class of irreducible stationary generalized Bratteli diagrams. In particular, it was shown that 

\begin{theorem} [\cite{Bezuglyi_Jorgensen_Karpel_Sanadhya_2023}] 
\label{Thm:inv1} 
Let $B(F) = B(V,E,>)$ be an ordered stationary generalized 
Bratteli diagram 
such that the matrix $A = F^T$ is infinite, irreducible, aperiodic, 
and positive recurrent. Let $\xi = (\xi_v)$ be a Perron-Frobenius 
right eigenvector for $A$ such that 
$\sum_{u\in V_0} \xi_u = 1$. Then the measure $\mu$ given by the formula
 $$
 \mu([\ol e(w, v)]) = \frac{\xi_v}{\lambda^{n}}, 
 $$
 where $[\ol e(w, v)]$ is the cylinder set which corresponds to a finite path $\ol e(w, v)$ that begins at $w \in V_0$ and ends at $v \in 
    V_n$, $n \in \N$,
is the unique probability  
$\varphi_B$-invariant measure that takes positive values 
on cylinder sets. 
\end{theorem}


\section{Measure extension for generalized Bratteli diagrams}\label{Sect:meas_ext_general}

In this section, we consider the procedure of a measure extension from a vertex subdiagram for generalized Bratteli diagrams. The proof is essentially the same as in \cite{AdamskaBezuglyiKarpelKwiatkowski2017}.

\begin{theorem}\label{Thm:Meas-ext-general} Let $B$ be a generalized Bratteli diagram $B =(V, E)$ with
incidence matrices $(F_n)$. Let $\ol B$ be a (standard or generalized) vertex subdiagram of $B$ determined
by a sequence $(W_n)$ of proper subsets $W_n \subset V_n$ for each $n \in \N$.
Let $\ol \mu $ be a probability measure on the path space
$X_{\ol B}$ of $\ol B$. 
Then 

\begin{equation}\label{Eq:mu_ext}
\widehat{\ol{\mu}}(\wh{X}_{\overline{B}}) = 1 + \sum_{n = 0}^{\infty}\sum_{v \in W_{n + 1}}
\sum_{w\in {W}'_{n}} {f}_{vw}^{(n)}  H_{w}^{(n)} 
\ol p_v^{(n+1)},
\end{equation}
where $W_{n}^{'} = V_{n} \setminus W_{n},\ n = 0, 1, 2, \ldots$.

In particular, the following statements are equivalent:

\noindent 
(i) 
$\displaystyle
\widehat{\ol{\mu}}(\wh{X}_{\overline{B}}) < \infty,$ 

\noindent
(ii) 
$\displaystyle
\sum_{n = 0}^{\infty}\sum_{v \in W_{n + 1}}
\sum_{w\in {W}'_{n}} {f}_{vw}^{(n)}  H_{w}^{(n)} 
\ol p_v^{(n+1)} < \infty. $

\end{theorem}

\begin{proof} To prove the theorem, fix $n$ and begin with equality \eqref{extension_method}. We have
$$
\ba
\wh{\overline{\mu}}( \wh{X}_{\overline{B}}^{(n + 1)}) = &\ 
\sum_{v \in W_{n + 1}}
\ol p_v^{(n+1)}  H_{v}^{(n + 1)} = 
\sum_{v \in W_{n + 1}} \ol p_v^{(n+1)} 
 \sum_{w \in V_{n}} {f}_{vw}^{(n)} H_{w}^{(n)} \\
= &\
\sum_{v \in W_{n + 1}} \ol p_v^{(n+1)} 
\sum_{w \in W_{n}} {f}_{vw}^{(n)} H_{w}^{(n)} + 
\sum_{v \in W_{n + 1}} \ol p_v^{(n+1)} 
\sum_{w \in W'_{n}} {f}_{vw}^{(n)} H_{w}^{(n)} \\
=& \ 
\sum_{w \in W_{n}}  H_{w}^{(n)} \sum_{v \in W_{n + 1}} \ol {f}_{vw}^{(n)} \ol p_v^{(n+1)}  + 
\sum_{v \in W_{n + 1}} 
\sum_{w \in W'_{n}} {f}_{vw}^{(n)}\ol p_v^{(n+1)}H_{w}^{(n)} \\
= &\ 
\sum_{w \in W_{n}}  H_{w}^{(n)} \ol p_w^{(n)} +
\sum_{v \in W_{n + 1}} 
\sum_{w \in W'_{n}} {f}_{vw}^{(n)}\ol p_v^{(n+1)} H_{w}^{(n)} \\
=&\  
\wh{\overline{\mu}} (\wh{X}_{\overline{B}}^{(n)}) +
\sum_{v \in W_{n + 1}} 
\sum_{w \in W'_{n}} {f}_{vw}^{(n)}\ol p_v^{(n+1)}H_{w}^{(n)}. 
\ea 
$$
Since
$$
\wh{\overline{\mu}}(\wh{X}_{\overline{B}})
= 1 + \sum_{n\geq 0} \left(\wh{\overline{\mu}} 
(\wh{X}_{\overline{B}}^{(n+1)})  - \wh{\overline{\mu}} 
(\wh{X}_{\overline{B}}^{(n)})\right), 
$$ 
we get the result.
\qedhere \end{proof}


\begin{remark}
Note that we can also compute the measure $\wh {\ov \mu}$ of any cylinder set $[\ov e]$ by carefully examining which cylinder subsets that end in the vertices of $\ov B$ are contained in $[\ov e]$. The sum of measures of these subsets will give us the measure of $[\ov e]$. In Section~\ref{Sect:Examples_DIO}, we present such computations for concrete examples of Bratteli diagrams.
\end{remark}

\section{Reducible Bratteli diagrams with infinitely many odometers}\label{Sect:DIO_meas_general}

Consider the following class of non-stationary 
reducible generalized Bratteli diagrams. In this section, we will focus on the study of their tail invariant measures. Every diagram in this class contains infinitely many odometers as subdiagrams, which are connected by single edges.

Let the Bratteli diagram $B = B_{IO}$ be defined by the sequence of incidence matrices 
\be\label{eq-m-x nostat DIO}
{F}_{n} =
\begin{pmatrix}
 a_{n}^{(1)} &      1 &      0 &      0 & \ldots & \ldots & \ldots & \ldots  \\
     0 &  a_{n}^{(2)} &      1 &      0 &      0 & \ldots & \ldots & \ldots  \\
     0 &      0 &  a_{n}^{(3)} &      1 &      0 &      0 & \ldots & \ldots  \\
     0 &      0 &      0 &  a_{n}^{(4)} &      1 &      0 &      0 & \ldots  \\
\ldots & \ldots & \ldots & \ldots & \ldots & \ldots & \ldots & \ldots  \\
\end{pmatrix},\qquad n \in \N_0,
\ee
where the natural numbers $a_{n}^{(i)} \geq 2$ for all $n \in \mathbb{N}_0$ and $i \in \mathbb{N}$.
The index $n$ points out at the $n$-th level of the diagram $B$,
and $i$ corresponds to the number of a vertex inside $V_n$.

The diagram $B_{IO}$ has a natural set of elementary vertex subdiagrams $\ol B_i$ consisting of vertical odometers where $i$ runs over the set $\N$. There are exactly $a_n^{(i)}$ edges connecting the vertices $i \in V_n$ and $i \in V_{n+1}$. 
We call $B$ the ``diagram of infinite odometers (DIO)''. 

Analyzing the paths space of (DIO) it is not hard to prove the following claim:

\begin{claim}\label{r2.24} 
The sets $\wh X_{\ov B_i}$, $i = 1, 2, \ldots$,
are pairwise disjoint and
$X_{B}= \bigsqcup_{i = 1}^{\infty}\wh X_{\ov B_i}$, where $X_{B}$ is the set of all infinite paths of (DIO).
\end{claim}
 
 The subdiagram $\ol B_i$ 
of $B$ admits a unique tail invariant probability measure 
$\ol \mu_i$ on the path space $X_{\ol B_i}$ such that for a 
cylinder set $[\ol e] = [e_0, \ldots, e_n], s(e_j) = r(e_j) =i$, we have
$$
\ol \mu_i ([\ol e]) = \frac{1}{a^{(i)}_0\ \cdots\ 
a^{(i)}_n}.
$$
The measure extension procedure applied to $\ol B_i$ 
gives us the measure $\wh{\ol \mu}_i$ on the tail invariant set 
$\wh X_{\ol B_i}$. It follows from Theorem 
\ref{Thm:Meas-ext-general} that 
$$
\wh{\ol \mu}_i(\wh X_{\ol B_i}) < \infty \ 
\Longleftrightarrow \ \sum_{n=0}^\infty \frac{H^{(n)}_{i+1}}
{a^{(i)}_0\ \cdots\ a^{(i)}_n} < \infty.
$$
Thus, it follows from the construction of $B_{IO}$ 
that there are infinitely many ergodic measures 
$\wh{\ol \mu}_i$ on the path space $X_B$. Some of them may
be finite, while others are infinite. We will give exact examples
below. Moreover, the measures  $\wh{\ol \mu}_i$ 
and $\wh{\ol \mu}_j$ are 
mutually singular ($i \neq j)$ because they are supported by 
non-intersecting tail invariant sets $\wh X_{\ov B_i}$ and $\wh X_{\ov B_j}$ (see Claim~\ref{r2.24}). 

Our goal is to show that there are no other ergodic measures. 

\begin{remark}\label{rem-theta vs mu}
Let $\theta$ be a finite tail invariant measure on 
$\wh X_{\ol B_i}$. Then there is a constant $C > 0$ such that 
$\theta = C \wh{\ol \mu}_i$. 

Indeed, let $C = \theta(X_{\ol B_i})$, then $\ol \theta :=
\theta|_{X_{\ol B_i}}= C \ol \mu_i$ by uniqueness of $\ov \mu_i$ on $X_{\ov B_i}$. By tail invariance, for every $n \in \mathbb{N}_0$ and every cylinder set $[e_0, \ldots, e_n]$ such that $r(e_n) = i \in V_{n+1}$ we have
$$
\theta([e_0, \ldots, e_n]) =  \frac{C}{a^{(i)}_0\ \cdots\ 
a^{(i)}_n} = C \wh{\ol \mu_i}([e_0, \ldots, e_n]).
$$
\end{remark}

\begin{thm}\label{thm:All-erg-inv-meas-DIO}
Let $\mc M$ be the family of measures obtained by normalization of measures $\wh{\ol\mu}_i$ 
such that $\wh{\ol \mu}_i(\wh X_{\ol B_i}) < \infty$. 
Then $\mc M$ coincides with the set of all ergodic probability 
tail invariant measures on the path space $X_B$ of the diagram 
$B$. 
\end{thm}

\begin{proof}
We first note that $X_B $ can be partitioned into the union of
tail invariant sets $\wh X_{\ol B_i}$. Every 
$\wh X_{\ol B_i}$ is an $F_\sigma$-set, hence it is Borel. 
Let $\theta$ be a finite tail invariant measure on $X_B$.
The support of $\theta$ is the union of some sets $\wh X_{\ol B_i}$. Then $\theta(\wh X_{\ol B_i}) < \infty$ for all 
such $i$'s. Let $\theta_i$ be the measure $\theta$ restricted
to the set $\wh X_{\ol B_i}$. By Remark \ref{rem-theta vs mu}, every $\theta_i$ is proportional to $\wh{\ol \mu}_i$, i.e.,
$\theta_i = C_i\wh{\ol \mu}_i$. We observe that $\sum_i C_i =
\theta(X_B)$. This means that $\theta$ is a 
linear combination of ergodic measures $\wh{\ol \mu}_i$.
\end{proof}

\begin{remark}
The same approach works in the case when the vertical
odometers are replaced with simple stationary 
standard Bratteli diagrams $\ol B_i$. As for odometers, we will 
have a unique ergodic probability measure $\ol\mu_i$
on the path space $X_{\ol B_i}$ defined in  
Theorem \ref{thm-BKMS2010}. Assuming that the extension
$\wh{\ol\mu}_i(\wh X_{\ol B_i})$ is finite, we get
that this measure is unique (up to
a constant). The same reasoning as in the proof of Theorem 
\ref{thm:All-erg-inv-meas-DIO} can be repeated. 
\end{remark}

\section{Some classes of generalized Bratteli diagrams with infinitely many odometers}\label{Sect:Examples_DIO}

In this section, we present some classes of stationary and non-stationary reducible generalized Bratteli diagrams with infinitely many odometers and find their sets of probability ergodic invariant measures. We give examples of diagrams that: (i) have a unique probability ergodic invariant measure, (ii) have countably many probability ergodic invariant measures, (iii) have no probability invariant measure, but possess an infinite $\sigma$-finite invariant measure that takes finite values on all cylinder sets, and  (iv) have no invariant measure that takes finite values on all cylinder sets.

\subsection{Stationary generalized Bratteli diagrams with infinitely many odometers.}
In this subsection, we describe all probability ergodic invariant measures for a class of stationary generalized Bratteli diagrams. We apply two 
methods: the construction of measure extension and the procedure of obtaining a measure from a positive eigenvector and eigenvalue. We show that these two 
approaches lead to the same ergodic invariant measures. These procedures also allow us 
to obtain infinite $\sigma$-finite ergodic invariant measures.

Theorems~\ref{Thm:measures_Exk}, \ref{Thm:Dio_k_eigenvectors}, and \ref{Thm:DIO_a_i_all_inv_m} show that for stationary reducible generalized Bratteli diagrams, unlike the case of standard diagrams, if a class of vertices is distinguished it doesn't necessarily mean that the corresponding ergodic invariant measure is finite. We will use the notation introduced in Section~\ref{Sect:DIO_meas_general}.

\begin{thm}\label{Thm:measures_Exk}
    Let $B$ be a stationary generalized Bratteli diagram
 with incidence matrix
 $$
F = \begin{pmatrix}
    a & 1 & 0 & 0 & 0 &\ldots\\
    0 & a-k & 1 & 0 & 0 &\ldots\\
    0 & 0 & a-k & 1 & 0 &\ldots\\
    0 & 0 & 0 & a-k & 1 &\ldots\\
    \vdots & \vdots & \vdots & \vdots & \vdots &\ddots
\end{pmatrix},
 $$
 where $a, k \in \mathbb{N}$ and $a - k > 1$. 
 Then there is a unique probability ergodic invariant measure $\mu$ on $B$ if and only if $k > 1$. 
 If $k = 1$ then there are no probability invariant measures on $B$. 
 
\end{thm}

To prove the theorem we will need the following lemma:

\begin{lemma}\label{Lemma:Heights_Exk}
    For every $n \in \mathbb{N}_0$, we have
\begin{equation}\label{Eq_Heights_Exk}
   H_i^{(n)} = (a - (k-1))^n \mbox{ for } i > 1.
\end{equation}
\end{lemma}

\begin{proof}[Proof of Lemma~\ref{Lemma:Heights_Exk}]
First notice that $H_i^{(n)} = H_2^{(n)}$ for all $n \in \mathbb{N}_0$ and $i > 1$.
We prove the formula~(\ref{Eq_Heights_Exk}) by induction. For all $i > 1$ we have 
$$
H_i^{(0)} = 1 \mbox{ and } H_i^{(1)} = (a - k) + 1.
$$
Assume that $H_i^{(n)} = (a - (k-1))^n$ for all $i > 1$. Then
$$
H_i^{(n+1)} = (a - k)H_i^{(n)} + H_{i+1}^{(n)} = (a - k + 1)(a - (k-1))^n =  (a - (k-1))^{n+1}. \qedhere
$$
\end{proof}

\begin{proof}[Proof of Theorem~\ref{Thm:measures_Exk}]
By Theorem~\ref{Thm:Meas-ext-general} and Lemma~\ref{Lemma:Heights_Exk}, we have    
$$
\wh{\ov \mu}_1 (\wh X_{\ov B_1}) = 1 + \sum_{n = 0}^{\infty} \frac{H_2^{(n)}}{a^{n + 1}} = 1 + \sum_{n = 0}^{\infty} \frac{(a - (k-1))^n}{a^{n + 1}} = 1 + \frac{1}{a}\sum_{n = 0}^{\infty} \left(1 - \frac{k-1}{a}\right)^n.
$$
Let $q = 1 - \frac{k-1}{a} \leq 1$. Assume that $k > 1$, so $q < 1$. Then
$$
\wh{\ov \mu}_1 (\wh X_{\ov B_1}) = 1 + \frac{1}{a}\sum_{n = 0}^{\infty} q^n = 1 + \frac{1}{a(1-q)} = 1 + \frac{a}{a(k-1)} = 1 + \frac{1}{k-1}.
$$

If $k = 1$, then $\wh{\ov \mu}_1 (\wh X_{\ov B_1}) = \infty$.

For any $k \geq 1$ and $i > 1$ we also have
$$
\wh{\ov \mu}_i (\wh X_{\ov B_i}) = 1 + \sum_{n = 0}^{\infty} \frac{H_{i+1}^{(n)}}{(a-k)^{n + 1}} = 1 + \sum_{n = 0}^{\infty} \frac{(a - k + 1)^n}{(a-k)^{n + 1}} = \infty.
$$
Thus, by Theorem~\ref{thm:All-erg-inv-meas-DIO}, for $k = 1$, there are no probability invariant measures on $B$, and for each $k > 1$, there is a unique probability ergodic invariant measure $\mu$ on $B$ which is an extension of the unique invariant measure from the first odometer.
\end{proof}

\begin{thm}\label{Thm:Dio_k_eigenvectors}
For $k > 1$, the unique (up to constant multiple) ergodic invariant measure $\mu$ on $B$ from Theorem~\ref{Thm:measures_Exk}
 can be obtained by using the 
 eigenvalue $\lambda = a$, the corresponding 
    eigenvector 
    \begin{equation}\label{eq:xi}
    \xi = (\xi_i) = \left(1, \frac{1}{k}, \frac{1}{k^2}, \ldots, \frac{1}{k^{i-1}}, \ldots\right)^T
    \end{equation}
    and the formula 
    \begin{equation}\label{Eq_mu_cyl}
   \mu([\ol e(w, v)]) = \frac{\xi_v}{\lambda^{n}}, 
\end{equation}
    where $\ol e(w, v)$ is a  
    finite path that begins at $w \in V_0$ and ends at $v \in 
    V_n$ (recall that we identify the vertices $v \in V_n$ with natural numbers $i \in \mathbb{N}$). 

For $k = 1$, there is an ergodic infinite $\sigma$-finite invariant measure $\mu$ which takes finite positive values on all cylinder sets. This measure can be obtained as the extension of the unique invariant measure from the first odometer as in Theorem~\ref{Thm:measures_Exk} or using the 
 eigenvalue $\lambda = a$, the corresponding
    eigenvector 
    $$
    \xi = (\xi_i) = \left(1, 1, 1, \ldots \right)^T
    $$
    and  formula~(\ref{Eq_mu_cyl}). 

\end{thm}

\begin{proof} First we prove that, for $k \geq 1$, the vector $\xi = (\xi_i)$ given by the formula~\eqref{eq:xi} is a right eigenvector for $A = F^T$ associated with the eigenvalue $\lambda = a$. Indeed, for any eigenvector $\eta = (\eta_i)$ and eigenvalue $\lambda$, we have $a\eta_1 = \lambda \eta_1$. If we take $\eta_1 = 1$, then $\lambda = a$. From the equality $\eta_1 + (a - k)\eta_2 = \lambda \eta_2$ we obtain $1 + (a - k)\eta_2 = a \eta_2$ and $\eta_2 = \frac{1}{k}$. It is easy to prove by induction that, for every $i \in \mathbb{N}$, we have $\eta_i = \frac{1}{k^{i-1}} = \xi_i$. 
Note that for $k = 1$ we obtain eigenvector $\xi = (1,1,1, \ldots)^T$.
By Theorem~\ref{BKMS_measures=invlimits}, the measure $\mu$ defined on each cylinder set by formula~(\ref{Eq_mu_cyl}) determines a tail-invariant measure on $B$ which takes finite values on cylinder sets. We show that this measure coincides with $\wh {\ov \mu}_1$ on all cylinder sets. First we prove that $\mu([{\ov e}^{(m)}(2)]) = \wh {\ov \mu}([{\ov e}^{(m)}(2)])$ for all $m \in \mathbb{N}$, where $[{\ov e}^{(m)}(2)]$ is any cylinder set of length $m$ which ends in vertex $2$. 

We have $\mu([{\ov e}^{(m)}(1)]) = \wh {\ov \mu}_1([{\ov e}^{(m)}(1)]) = \frac{1}{a^{m}}$ for all $m \in \mathbb{N}_0$. Then we can check that
$$
\wh {\ov \mu}_1([{\ov e}^{(0)}(2)]) = \sum_{n = 0}^{\infty} \frac{(a-k)^{n}}{a^{n+1}}  =  \frac{1}{a}\cdot \frac{1}{1 - \frac{a - k}{a}} = \frac{1}{k} = \mu([{\ov e}^{(0)}(2)]).
$$
Similarly, since the diagram is stationary, we obtain
$$
\wh {\ov \mu}_1([{\ov e}^{(1)}(2)]) = \sum_{n = 0}^{\infty} \frac{(a-k)^{n}}{a^{n+2}}  =  \frac{1}{a^2}\cdot \frac{1}{1 - \frac{a - k}{a}} = \frac{1}{ak} =  \mu([{\ov e}^{(1)}(2)]).
$$
In general, for every $m \in \mathbb{N}_0$,
$$
\wh {\ov \mu}_1([{\ov e}^{(m)}(2)]) = \sum_{n = 0}^{\infty} \frac{(a-k)^{n}}{a^{n + m + 1}}  =  \frac{1}{a^{m+1}}\cdot \frac{1}{1 - \frac{a - k}{a}} = \frac{1}{a^mk} = \mu([{\ov e}^{(m)}(2)]).
$$
Now we can use the knowledge about the measures of cylinder sets that end in the second vertex and similarly show that  
$$
\wh {\ov \mu}_1([{\ov e}^{(0)}(3)]) = \sum_{n = 0}^{\infty} \frac{(a-k)^{n}}{ka^{n+1}}  =  \frac{1}{ak}\cdot \frac{1}{1 - \frac{a - k}{a}} = \frac{1}{k^2} = \mu([{\ov e}^{(0)}(3)]).
$$ 
In general, we can prove by induction on $i$ that for every $i \in \mathbb{N}$ and $m \in \mathbb{N}_0$ we have
$$
\wh {\ov \mu}_1([{\ov e}^{(m)}(i+1)]) = \sum_{n = 0}^{\infty} \frac{(a-k)^{n}}{k^{i-1}a^{n+m+1}}  =  \frac{1}{a^{m+1}k^{i-1}}\cdot \frac{1}{1 - \frac{a - k}{a}} = \frac{1}{k^{i}a^{m}} = \mu([{\ov e}^{(m)}(i+1)]). \qedhere
$$
\end{proof}

\begin{remark}
For $i > 1$ and $k \in \mathbb{N}$, the measure $\wh{\ov \mu}_i$ does not attain finite values on all cylinder sets. In particular, for every $i > 1$ we have
    $$
    \wh{\ov \mu}_i([{\ov e}^{(0)}(i+1)]) = \infty.
    $$
    Indeed, we obtain
    $$
    \wh{\ov \mu}_i([{\ov e}^{(0)}(i+1)]) = \sum_{n = 0}^{\infty} \frac{(a-k)^{n}}{(a-k)^{n+1}}  =  \sum_{n = 0}^{\infty} \frac{1}{a-k} = \infty.
    $$
\end{remark}

More generally, the following theorem holds:
\begin{theorem}\label{Thm:DIO_a_i_all_inv_m}
Let $B$ be a stationary generalized Bratteli diagram with  the incidence matrix
$$
F = \begin{pmatrix}
    a_1 & 1 & 0 & 0 &\ldots\\
    0 & a_2 & 1 & 0 & \ldots\\
    0 & 0 & a_3 & 1 & \ldots\\
    0 & 0 & 0 & a_4 & \ldots\\
    \vdots & \vdots & \vdots & \vdots &\ddots
\end{pmatrix},
 $$
 where $a_1 > a_k$ for all $k \geq 2$.
 Then there is an invariant measure $\mu$ on $B$ generated by the right eigenvector $\mathbf{x}$ of $A = F^T$ associated to the eigenvalue $\lambda = a_1$, where
 $$
 \mathbf{x}^T = \left(1, \; \frac{1}{a_1 - a_2}, \; \frac{1}{(a_1 - a_2)(a_1 - a_3)}, \; \ldots, \quad \frac{1}{(a_1 - a_2)(a_1 - a_3)\cdots (a_1 - a_n)},\; \ldots \; \right).
 $$ 
This measure can be also obtained as an extension of 
 the probability measure from the odometer that passes through the diagram's first vertex. The measure $\mu$ is finite if and only if 
 $$
 \sum_{n = 2}^{\infty} \prod_{k = 2}^n \frac{1}{a_1 - a_k} < \infty.
 $$
\end{theorem}

\begin{proof}
It is straightforward to check that $\mathbf{x}^T F = a_1 \mathbf{x}$. We show that the measure $\mu$ generated by $a_1$ and $\mathbf{x}$ can be obtained as an extension of the unique invariant probability measure $\ov \mu_1$ from the first odometer.
    
We have 
$$
\mu([{\ov e}^{(m)}(1)]) = \wh {\ov \mu}_1([{\ov e}^{(m)}(1)]) = \frac{1}{a_1^{m}}$$ for all $m \in \mathbb{N}_0$. Then for every $m \in \mathbb{N}_0$,
    $$
    \wh {\ov \mu}_1([{\ov e}^{(m)}(2)]) = \sum_{n = 0}^{\infty} \frac{a_2^n}{a_1^{n + m + 1}} =  \frac{1}{a_1^{m+1}} \cdot \frac{1}{1 - \frac{a_2}{a_1}} = \frac{1}{(a_1 - a_2)a_1^m} =\mu([{\ov e}^{(m)}(2)])
    $$
    and
        $$
        \ba
    \wh {\ov \mu}_1([{\ov e}^{(m)}(i+1)]) = &\ \sum_{n = 0}^{\infty} \frac{a_{i+1}^n}{a_1^{n + m + 1}\prod_{j = 1}^i(a_1-a_j)}\\ = &\ \frac{1}{a_1^{m + 1}\prod_{j = 1}^i(a_1-a_j)} \cdot \frac{1}{1 - \frac{a_{i+1}}{a_1}}\\  =  &\ \frac{1}{\prod_{j = 1}^{i+1}(a_1-a_j)a_1^m} =\mu([{\ov e}^{(m)}(i+1)]). \qedhere
    \ea 
    $$ 
 \end{proof}

\begin{remark}
Notice that in the case when there is $m \in \mathbb{N}$ such that $a_m > a_k$ for all $k > m$ then there is an eigenvector
$$
\textbf{x}^T = \left(0, \; \ldots, \; 0, \; 1, \; \frac{1}{a_m - a_{m+1}}, \; \ldots, \; \frac{1}{(a_m - a_{m+1})\cdots (a_m - a_{m + n})}, \; \ldots  \; \right)
$$
corresponding to the eigenvalue $\lambda = a_m$. These eigenvalue and eigenvector generate an invariant measure which is the extension of a probability measure from the odometer that passes through the vertex $m$ of the diagram.
\end{remark}

The following example was considered in \cite{Bezuglyi_Jorgensen_Karpel_Sanadhya_2023} and presents a stationary generalized Bratteli diagram with infinitely many odometers such that there does not exist any tail-invariant measure on 
$X_B$ that assigns finite values to all cylinder sets. Indeed, it is straightforward to check that the measure extension from every odometer is infinite, moreover, for every $i \in \mathbb{N}$, we have 
$$
\wh{\ov \mu}_i ([{\ov e}^{(0)}(i+1)]) = \infty.
$$

\begin{example}\label{example}
Let $B = B(F)$ be a generalized stationary Bratteli 
diagram as shown in Figure \ref{Fig:no measure} and given by 
$\mathbb{N} \times \mathbb{N}$ incidence matrix 
\begin{equation}\label{matrix_no_prob}
    F = \begin{pmatrix}
2 & 1 & 0 & 0 & \ldots\\
0 & 3 & 1 & 0 & \ldots\\
0 & 0 & 4 & 1 & \ldots\\
0 & 0 & 0 & 5 & \ldots\\
\vdots & \vdots & \vdots & \vdots & \ddots\\
\end{pmatrix}.
\end{equation} Then there does not exist any tail-invariant measure on 
$X_B$ that assigns finite values to all cylinder sets.

\begin{figure}[hbt!]
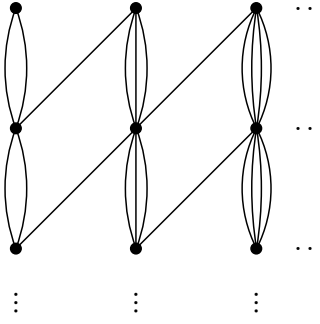

\unitlength=.8cm
\begin{graph}(7,6)
 \roundnode{V11}(2,5)
 \roundnode{V12}(4,5)
 \roundnode{V13}(6,5) 
 \roundnode{V21}(2,3)
 \roundnode{V22}(4,3)
 \roundnode{V23}(6,3)
 \roundnode{V31}(2,1)
 \roundnode{V32}(4,1)
 \roundnode{V33}(6,1)
  %
 %
 \graphlinewidth{0.025}

 \bow{V21}{V11}{0.09}
  \bow{V21}{V11}{-0.09}
    \edge{V21}{V12}
     
 \bow{V22}{V12}{0.09}
 \bow{V22}{V12}{-0.09}
 \edge{V22}{V12}
 
 \bow{V23}{V13}{0.12}
  \bow{V23}{V13}{-0.12}
   \bow{V23}{V13}{0.04}
  \bow{V23}{V13}{-0.04}
 \edge{V22}{V13}

 \bow{V31}{V21}{0.09}
 \bow{V31}{V21}{-0.09}
   \edge{V31}{V22}
 \bow{V32}{V22}{0.09}
 \bow{V32}{V22}{-0.09}
 \edge{V32}{V22}

  \bow{V33}{V23}{0.12}
  \bow{V33}{V23}{-0.12}
   \bow{V33}{V23}{0.04}
  \bow{V33}{V23}{-0.04}
 \edge{V32}{V23}
 
    \freetext(6.9,5){$\ldots$}
  \freetext(6.9,3){$\ldots$}
    \freetext(6.9,1){$\ldots$}
    \freetext(2,0.1){$\vdots$}
  \freetext(4,0.1){$\vdots$}
    \freetext(6,0.1){$\vdots$}
\end{graph}
\caption{A stationary generalized Bratteli diagram with no invariant measure that takes finite values on all cylinder sets (illustration to Example~\ref{example}).}\label{Fig:no measure}
\end{figure}
\end{example}

\subsection{Non-stationary generalized Bratteli diagrams with infinitely many odometers.}
Now, we modify the diagram given in 
Example \ref{example} and consider a non-stationary
generalized Bratteli diagram $B$ defined by a sequence of 
natural numbers $\{a_{i} : i \in \N_0 \}$.
Without 
loss of generality, we can assume that $a_n \geq 2$ for all
$n$.
The diagram $B$ consists of an infinite sequence of
non-stationary odometers 
connected with the neighboring odometer by single edges.
More precisely, let  $V_{n} = \N$
and for any $n \in \mathbb{N}_0$, the incidence matrix $F_{n}$ has the form
\be\label{eq-m-x DIO}
{F}_{n} =
\begin{pmatrix}
 a_{n} &      1 &      0 &      0 & \ldots & \ldots & \ldots & \ldots  \\
     0 &  a_{n} &      1 &      0 &      0 & \ldots & \ldots & \ldots  \\
     0 &      0 &  a_{n} &      1 &      0 &      0 & \ldots & \ldots  \\
     0 &      0 &      0 &  a_{n} &      1 &      0 &      0 & \ldots  \\
\ldots & \ldots & \ldots & \ldots & \ldots & \ldots & \ldots & \ldots  \\
\end{pmatrix}.
\ee

Fix $i \geq 1$, set $W_{n} = \left\{ i \right\}$, $n = 1, 2, 
\ldots$ and define the subdiagram ${\overline{B}}_i = (\ol W, {\ol E})$ as above with the only difference that 
the set $\ol E_n$ is formed now by $a_n$ edges connecting 
the vertices $i \in V_n$ and $i\in V_{n+1}$. 
The unique tail invariant probability measure
$\overline{\mu} = \overline{\mu}_i$ on
${\overline{B}}_i = (\ol W, {\ol E})$ is given by the formula
$\overline{\mu}([\ol e]) = \dfrac{1}{a_{0}\ \cdots\ a_{n}}$ 
where $r(\ol e) = i \in V_{n+1}$. 

By definition of the diagram, $H^{(n)} = H^{(n)}_i$ 
for all  $n$ and  $i$ (as usual, we set $H^{(0)}_i =1$). 
Then $H^{(n + 1)} = H_{i}^{(n + 1)}  = (a_{n} + 1) H^{(n)}$ 
which implies that 
$$H^{(n+1)} = (a_{0} +1) \ \cdots \ (a_{n} + 1), \quad n \in 
\N_0.
$$ 

\begin{thm}\label{Thm:non-stat-DIO-ex-erg-mu}
For every $i$, the extension $\wh{\ol \mu}_i$ of $\ov \mu_i$ is finite if and only if 
$$
\sum_{n \geq 0} \frac{1}{a_n} < \infty.
$$
If the extensions $\wh{\ov \mu}_i$ for $i \in \mathbb{N}$ are finite then the set of all ergodic finite tail invariant measures
on the path space of the diagram $B$ coincides with $\{\wh{\ol \mu}_i \ : i \in \N\}$.
\end{thm}

\begin{proof}
Using Theorem \ref{Thm:Meas-ext-general}, 
we compute $\wh{\ol \mu}_i (\wh {X}_{\ov B_i})$ as follows:
$$
\ba
\wh{\ol \mu}_i(\wh X_{\ov B_i}) = &\ 
1 + \sum_{n \geq  0} \sum_{w \in {W}'_{n}} 
{f}_{iw}^{(n)} H_{w}^{(n)} \frac{1}{a_0 \ \cdots\ a_{n}}\\
= &\ 1 + \sum_{n \geq 1}\frac{{(a}_{0} + 1)  \cdot \ldots \cdot (a_{n-1} + 1) }{a_{0}  \cdot \ldots \cdot a_{n}}\\
= &\ 
1 + \sum_{n \geq 0} \frac{1}{a_{n}} \cdot \left[ 
\prod_{j = 0}^{n-1}\left(1 + \frac{1}{a_{j}}\right) \right] \\ 
\leq &\ 
1 + \prod_{j = 0}^{\infty}\left(1 + \frac{1}{a_{j}}\right) \cdot 
\sum_{n = 0}^{\infty} \frac{1}{a_{n}}.
\ea
$$
Thus, the extensions $\wh{\ov \mu}_i$ are finite if and only if $\sum_{n \geq 0} \frac{1}{a_n} < \infty.$
By Theorem~\ref{thm:All-erg-inv-meas-DIO}, if the extensions $\wh{\ov \mu}_i$ are finite then the measures $\{\wh{\ol\mu}_i : i \in \mathbb{N}\}$ form the set of all probability ergodic invariant measures on $B$. 
\end{proof}

\section{Vershik maps on generalized Bratteli diagrams with infinitely many odometers}\label{Sect:Vmap}

In this section, we consider different orders on reducible generalized Bratteli diagrams with infinitely many odometers (see Section~\ref{Sect:DIO_meas_general}). 
We present different classes of orders on such diagrams and study whether the corresponding Vershik map can be extended to a homeomorphism.

Recall that each vertex $v$ of (DIO) is determined by the pair $(n, i)$, where $n \geq 0$ is the level such that $v \in V_{n}$ and $i = 1, 2, \ldots$ is the vertex number of $v$ inside $V_{n}$. We will write $v = (n, i)$. For $v = (n,i) \in V \setminus V_0$, we have $r^{-1}(v) = \{e_{1}, e_{2}, \ldots, e_{a_n^{(i)}}, f\}$, 
where $e_{1}, e_{2}, \ldots, e_{a_n^{(i)}}$ are edges between $v$ and $w =(n-1, i)$, 
whereas $f$ is the edge between $v$ and $w = (n-1, i+1)$. 

Let $\omega = \{ \omega_{v}, \; v \in V \setminus V_{0} \}$ 
be an order of (DIO), i.e., for every 
$v \in  V \setminus V_{0}$,  ${\omega}_{v}$ 
is a linear order on the set $r^{-1}(v)$. 
Let ${\overline{e}}_{v}$ and $\widehat{e}_{v}$ 
be the minimal and the maximal edges in
$r^{-1}(v)$ respectively. 
We will say that the order ${\omega}_{v}$ is \textit{left} if ${\overline{e}}_{v} = f$, is \textit{right} if $\widehat{e}_{v} = f$ and is \textit{middle} if
${\overline{e}}_{v}$ and $\widehat{e}_{v}$ are different from $f$. 
According to Claim \ref{r2.24}, every maximal (minimal) infinite path 
$\overline{e} = \{ {e'}_{1}, {e'}_{2}, \ldots \}$ 
is contained in a unique $\wh X_{\ov B_i}$, $i \geq 1$. 
We will say that the odometer $\ov B_i$ is 
\textit{finite-right (finite-left)} 
if there exists a number $n_{i}$ such that the orders
${\omega}_{v}$ are not right (not left) for $n \geq n_{i}$, $v = (n, i)$.

Since each $\ov B_i$ is an odometer, the set $\wh X_{\ov B_i}$ consists of all paths that eventually pass vertically through vertex $i$. Since two different minimal paths cannot pass through the same vertex, we obtain that the set $\wh X_{\ov B_i}$ can contain not more than one infinite minimal path. If $\wh X_{\ov B_i}$ does contain an infinite minimal path then this path is eventually vertical, hence the odometer $\ov B_i$ should be finite-left. Similarly, the set $\wh X_{\ov B_i}$ contains the unique maximal infinite path if and only if the odometer $\ov B_i$ is finite-right (we denote $i \in I_{fr}$). We summarize the above remarks in the following proposition:

\begin{prop}\label{r2.25} 
The set $\wh X_{\ov B_i}$ contains the unique maximal (minimal) infinite path $\widehat{e}_{i, \max}$ (${\overline{e}}_{i, \min}$) iff the odometer $\ov B_i$ is finite-right (finite-left). For the remaining 
odometers, 
$\wh X_{\ov B_i} \cap X_{\max} = \emptyset$ 
($\wh X_{\ov B_i} \cap X_{\min} = \emptyset$).
\end{prop}

Let $\varphi_{\omega} \colon (X_{B} \setminus X_{\max}) \rightarrow (X_{B} \setminus X_{\min}$) 
be the Vershik map defined by the order $\omega$. We now consider different orders $\omega$ on $B$ and examine when $\varphi_{\omega}$ can be extended to a homeomorphism. Since $B$ is not stationary, strictly speaking, we cannot consider stationary orders on $B$. We call an order on $B$ \textit{quasi-stationary} if for every vertex $i$ we fix whether it is left-ordered, right-ordered, or middle, and the type or order depends only on the vertex number but not on the level to which vertex belongs.

\begin{thm}\label{Thm:orders-DIO} 
Let $I_{fr}$ ($I_{fl}$) be the set of all $i \in \mathbb{N}$ 
such that $\ov B_i$ is finite-right (finite-left). 
Then

(i)  $\varphi_{\omega}$ extends to a Borel Vershik map on $X_B$
if and only if
$$
|I_{fr}| = |I_{fl}|;
$$

(ii) if 
$$
I_{fr} = I_{fl} = \emptyset,
$$ then $\varphi_{\omega}$ is a homeomorphism of $X_B$;

(iii) for any quasi-stationary order $\omega$ on $B$, the map $\varphi_{\omega}$ cannot extend to a Vershik homeomorphism of the whole path-space $X_B$.
\end{thm}

\begin{proof}

Part (i) is obvious since equality $|I_{fr}| = |I_{fl}|$ means that $|X_{max}| = |X_{min}|$. Since the sets $X_{max}$ and $X_{min}$ are always closed, the condition $|X_{max}| = |X_{min}|$ is necessary and sufficient for $\varphi_{\omega}$ to be extended to a Borel bijection of the whole space $X_B$.

(ii) If $I_{fr} = I_{fl} = \emptyset$, then $\varphi_{\omega}$ is a homeomorphism of the whole path space $X_B$. 

(iii) Note that the case $I_{fr} = I_{fl} = \emptyset$ is not possible for a  quasi-stationary order. Indeed, if the vertex $i \in \mathbb{N}$ is right-ordered then $i \in I_{fl}$ and there is a vertical minimal path that passes through the vertex $i$ on each level. Similarly, if the vertex $i$ is left-ordered then $i \in I_{fr}$ and there is a vertical maximal path passing through the vertex $i$. If the order $\omega_i$ of the vertex $i$ is middle then there are both infinite maximal path and infinite minimal path passing through the vertex $i$.
Assume that the map $\varphi_{\omega}$ can be extended to a Vershik homeomorphism on $X_B$. Then we have $|I_{fr}| = |I_{fl}|$. 
Moreover, since the order is quasi-stationary, we have $|I_{fr}| = |I_{fl}| = \aleph_0$. We will call a cylinder set maximal (minimal) if it corresponds to a maximal (minimal) finite path.

First, assume that $i \in I_{fr}$ and $i > 1$. 

Suppose $i-1 \in I_{fr} \cap I_{fl}$. Then for every $n$, every maximal cylinder set of length $n$ which ends in $i$ contains a non-maximal cylinder subset $[\ov g_1]$ of length $n+1$ which also ends in $i$ and a non-maximal cylinder subset $[\ov g_2]$ of length $n+1$ which ends in $i - 1$. The image $\varphi_w([\ov g_2])$ belongs to the minimal cylinder set of length $n$ which passes through vertex $i-1$ on each level. The image $\varphi_w([\ov g_1])$ belongs to the minimal cylinder set of length $n$ which ends in $i$ or $i+1$ and never passes through vertex $i-1$. Thus, we cannot extend $\varphi_{\omega}$ continuously to the maximal path which passes through vertex $i$.

If $i - 1 \in I_{fl} \setminus I_{fr}$, then the maximal edge which ends in $i-1$ is not vertical.  For every $n$, every maximal cylinder set of length $n$ which passes through vertex $i$ contains a non-maximal cylinder subset $[\ov g_1]$ of length $n+1$ which also ends in $i$ and a non-maximal cylinder subset $[\ov g_2]$ of length $n+2$ which passes through the maximal edge between vertices $i$ on level $n$ and $i-1$ on level $n+1$, and then through a non-maximal edge between vertex $i-1$ on level $n+1$ and vertex $i-1$ on level $n+2$. Again, the image $\varphi_w([\ov g_1])$ belongs to the minimal cylinder set of length $n$ which ends in $i$ or $i+1$, $\varphi_w([\ov g_2])$ belongs to the minimal cylinder set of length $n$ which passes through vertex $i-1$ on each level.

Thus, if $i \in I_{fr}$ and $i - 1 \in I_{fl}$ then $\varphi_{\omega}$ cannot be extended to a homeomorphism. It remains to notice that such a situation will always occur for any quasi-stationary order. Indeed, we have $|I_{fr}| = |I_{fl}| = \aleph_0$, hence we have infinitely many vertical infinite minimal paths and infinitely many vertical infinite maximal paths. Let $j \in \mathbb{N}$ be a vertex through which passes an infinite minimal path. Then there exists a vertex $k > j$ through which passes an infinite maximal path. Thus, we will necessarily find a natural number $i \in (j,k]$ such that $i \in I_{fr}$ and $i - 1 \in I_{fl}$.
\end{proof}

\medskip
\textbf{Acknowledgements.} 

We are very grateful to our colleagues, 
especially, P. Jorgensen, T. Raszeja, S. Sanadhya, M. Wata for valuable discussions. S.B. and O.K. are also grateful to the Nicolas Copernicus 
University in Toru\'n for its hospitality and support. S.B. acknowledges the hospitality of AGH University of Krak\'ow during his visit. O.K. is supported by the NCN (National Science Center, Poland) Grant 2019/35/D/ST1/01375 and by the program ``Excellence initiative - research university'' for the AGH University of Krak\'ow.
We thank the referees for their detailed comments on the paper which helped us to improve the exposition.

\bibliographystyle{alpha}
\bibliography{More_examples_on_GBD_and_measure_extension}

\end{document}